\documentclass{amsart}
\usepackage{amssymb,amsmath,ifthen,cite,graphicx,tikz}
\usepackage{tabu} % For inputting math in tabular.

\long\def\skipit#1{} %can skip over \par

\numberwithin{equation}{section}
\numberwithin{figure}{section}
\numberwithin{table}{section}

\newtheorem{thm}{Theorem}[section]

\newtheorem{lem}[thm]{Lemma}
\theoremstyle{definition}
\newtheorem{definition}[thm]{Definition}
\newtheorem{eg}[thm]{Example}
\newtheorem{conj}[thm]{Conjecture}

\def\fl#1{\lfloor{#1}\rfloor}
\def\dfl#1{\left\lfloor{#1}\right\rfloor}
\def\bgg#1{\biggl({#1}\biggr)}
\def\all#1{\qquad\textrm{for all $#1$}}

\def\tW{\tilde{W}}
\def\R{\mathbb{R}}
\def\mR{R}
\def\<{\;<\;}
\def\={\;=\;}
\def\>{\;>\;}
\def\eqrl{\quad\Longleftrightarrow\quad}
\newcommand{\rmand}{\qquad\hbox{and}\qquad}

\usepackage{calc}
\newcounter{hours}
\newcounter{minutes}
\newcommand{\printtime}{
\setcounter{hours}{\time/60}
\setcounter{minutes}{\time-\value{hours}*60}
\ifthenelse{\value{hours}<10}{0}{}\thehours:\hskip0pt
\ifthenelse{\value{minutes}<10}{0}{}\theminutes}

\usepackage{xcolor}

\begin{document}

\title{Root geometry of polynomial sequences I: Type $(0,1)$}

\author[J.L. Gross, T. Mansour, T.W. Tucker, and D.G.L. Wang]{Jonathan L. Gross
}
\address{
Department of Computer Science  \\
Columbia University, New York, NY 10027, USA; \quad
email: gross@cs.columbia.edu
}
\author[]{Toufik Mansour
}
\address{
Department of Mathematics  \\
University of Haifa, 31905 Haifa, Israel;  \quad
email: tmansour@univ.haifa.ac.il}
\author[]{Thomas W. Tucker
}
\address{
Department of Mathematics  \\
Colgate University, Hamilton, NY 13346, USA; \quad
email: ttucker@colgate.edu
}
\author[]{David G.L. Wang
}
\address{
School of Mathematics and Statistics  \\
Beijing Institute of Technology, 102488 Beijing, P. R. China;  \quad
email: david.combin@gmail.com}


\begin{abstract}   
This paper is concerned with the distribution in the complex plane of the roots of a polynomial sequence $\{W_n(x)\}_{n\ge0}$ given by a recursion $W_n(x)=aW_{n-1}(x)+(bx+c)W_{n-2}(x)$, with $W_0(x)=1$ and $W_1(x)=t(x-r)$, where $a>0$, $b>0$, and $c,t,r\in\mathbb{R}$.  Our results include proof of the distinct-real-rootedness of every such polynomial $W_n(x)$, derivation of the best bound for the zero-set $\{x\mid W_n(x)=0\ \text{for some $n\ge1$}\}$, and determination of three precise limit points of this zero-set. Also, we give several applications from combinatorics and topological graph theory.
\end{abstract}

\maketitle

%\date{}
%
%\begin{flushleft}   \vskip-24pt
%\textbf{Version:\quad\printtime\quad\today}
%\end{flushleft}

\bigskip

\section{\large Introduction}
Gian-Carlo Rota~\cite{Rota85} has said of the ubiquity of zeros of polynomials in combinatorics
\begin{center}
{\bf ``Disparate problems in combinatorics ... do have at least one common feature:\\
their solution can be reduced to the problem of finding the roots of some\\
polynomial or analytic function.''}
\end{center}
One such reduction is due to Newton's inequality, which implies that every real-rooted polynomial is log-concave.
As observed by Brenti~\cite{Bre89,Bre94}, polynomials that arise from combinatorial problems often turn out to be real-rooted.  

Given a sequence of polynomials $\{W_n(x)\}_{n\ge0}$, we refer to the distribution of the set of zeros, taken over all $n$, as the \emph{root geometry} of that sequence.  General information for the root geometry of polynomials, especially the geometry of non-real roots, is given by Marden \cite{Mar85B}; see also \cite{Obr03B,Pra10B}.

This research arose during efforts by the present authors to affirm a quarter-century old conjecture (abbr. the LCGD conjecture) that the genus distribution (or genus polynomial, equivalently) of every graph is log-concave~\cite{FGS89}. Although it was conjectured by Stahl \cite{St97} that genus polynomials are real-rooted, Chen and Liu \cite{CL10} proved otherwise. Subsequently, various genus polynomials have been shown to have complex roots.  Of course, this separates the problem of determining which graphs have real-rooted genus polynomials from trying to prove the  LCGD conjecture.

After unexpected success~\cite{GMTW14} in proving the real-rootedness of the genus polynomials of iterated claws, we attempted the real-rootedness of genus polynomials for iterated $3$-wheels \cite{PKG14}. The \textit{iterated 3-wheel} $W_3^n$ is the graph obtained from the cartesian product $C_3\Box P_{n+1}$, where $P_k$ is a path graph with $k$ vertices, by contracting a 3-cycle $C_3$ at one end of the product to a single vertex.  By a preprocess of normalization,  we transformed the problem equivalently into the following conjecture:
\begin{conj}\label{conj:wheel}
Let $W_0(x)=1/27$,\; $W_1(x)=1+7x$,\; $W_2(x)=1+139x+1120x^2+468x^3$,\; and
$$W_n(x) ~=~ (1+144x)W_{n-1}(x)+54x(2-29x+306x^2)W_{n-2}(x)-5832x^3(1-11x)W_{n-3}(x),$$
\indent for $n\ge3$. Then each of the polynomials $W_n(x)$ is real-rooted.
\end{conj}

Real-rootedness of the genus polynomials of iterated 3-wheels $W_3^n$ was confirmed by brute force computation for all $n\le 220$. The complications encountered led us to consider the more general problem for polynomial sequences defined by a general linear recurrence of degree $3$, with polynomial coefficients.  As one may imagine, the difficulty did not decrease. This led us to some recurrences of degree $2$. In particular, let $W_n(x)$ be a sequence of polynomials satisfying the recursion 
\begin{align}
W_n(x)~&=~ A(x)W_{n-1}(x)+B(x)W_{n-2}(x),\label{rec:W}
\end{align}
for $n\ge 2$, where $A(x)$ and $B(x)$ are polynomials, $W_0(x)$ is a constant, and $W_1(x)$ is a linear polynomial. When the polynomials $A(x)$ and $B(x)$ have degrees $k$ and $\ell$, respectively, we call the sequence $\{W_n(x)\}$ defined by \eqref{rec:W}  \emph{a recursive polynomial sequence of type~$(k,\ell)$}.

Classical bounds on the roots of a polynomial are given in terms of its coefficients. Examples include the Fujiwara bound \cite{Fuj16}, the Cauchy bound \cite{Cau1829}, and the Hirst-Macey bound \cite{HM97}. More bounds and also some background are given by Rahman and Schmeisser \cite{RS02}, where the reader may also find, for instance, Rouch\'e's theorem, Landau's inequality, and the Laguerre-Samuelson inequality, subject to bounding the roots of a polynomial. Conversely, the real-rooted polynomials with all roots in a prescribed interval have been characterized in terms of positive semi-definiteness of related Hankel matrices (see Lasserre \cite{Las02}). 

This paper is primarily concerned with the root geometry of a sequence of recursive polynomials of type $(0,1)$.

\bigskip
\section{\large Main Results and Examples}

As a preliminary, we consider a recursive polynomial sequence of type $(0,0)$, that is, one in which the polynomials $A(x)$ and $B(x)$ are constants, $A$ and $B$.  This serves as a bridge to considering a recursive sequence of polynomials of types in which $A(x)$ and $B(x)$ have other degree combinations.

\begin{lem}\label{lem:00}
Let $A,B\in\mathbb{R}$ with $A\neq0$.  Let $\{W_n\}_{n\ge0}$ be a sequence of real numbers satisfying the initial condition $W_0=1$ and the recursion $W_n=AW_{n-1}+BW_{n-2}$. Writing $\Delta=A^2+4B$ and $g^\pm=(2W_1-A\pm\sqrt{\Delta})/2$, we have
\begin{equation}\label{sol:W}
W_n=\begin{cases}
{\displaystyle \bgg{1+{n(2W_1-A)\over A}}\bgg{{A\over 2}}^n},&\textrm{ if $\Delta=0$};\\[10pt]
{\displaystyle {g^+(A+\sqrt{\Delta}\,)^n-g^-(A-\sqrt{\Delta}\,)^n
\over 2^{n}\sqrt{\Delta}}}
,&\textrm{ if $\Delta\neq0$}.
\end{cases}
\end{equation}
In particular, if $Re^{i\theta}$ is the polar representation of $A+\sqrt{\Delta}$, then we have
\begin{equation}\label{sol:W:Delta<0}
W_n=\bgg{R\over2}^n\bgg{\cos{n\theta}+{\sin{n\theta}\over \sqrt{-\Delta}}},
\qquad\textrm{if $\Delta<0$}.
\end{equation}
\end{lem}
\begin{proof}
The solution~(\ref{sol:W}) to Recursion~(\ref{rec:W}) can be found in elementary textbooks;  for more extensive discussion, see Kocic and Ladas~\cite{KL93}.  Note that when $A+\sqrt{\Delta}=Re^{i\theta}$, we have $A-\sqrt{\Delta}=Re^{-i\theta}$, since $\sqrt{\Delta}$ is either purely real or purely imaginary. Then, the expression~(\ref{sol:W:Delta<0}) can be obtained from~(\ref{sol:W}) directly.
\end{proof}

For instance, the Fibonacci sequence $\{f_n\}_{n\ge0}$ is defined by the recursion 
$$f_n ~=~ f_{n-1}+f_{n-2},$$
with $f_0=f_1=1$.  With $A=B=W_1=1$ (hence, $\Delta=5$ and $g^\pm=(1\pm\sqrt{5})/2$), Lemma~\ref{lem:00}  gives Binet's formula, as expected:
$W_n={1\over\sqrt{5}}\left((g^+)^{n+1}-(g^-)^{n+1}\right)$.
Thus, we see how Lemma~\ref{lem:00} creates conditions forrecursive sequences of type~$(0,0)$, under which the root geometry problem becomes easy.  
\medskip

\subsection{Main result}
The aim of this paper is to describe the root geometry of all polynomial sequences of type $(0,1)$.  In order to formulate the main results of this paper, we use the following terminology. 
\smallskip

\begin{definition}
The \emph{zero-set} of a polynomial $p(x)$ is defined to be the set of all its roots.  It is said to be \emph{distinct-real-rooted} if all its roots are distinct and real.  
\end{definition}
\smallskip

\begin{definition}\label{def:interlacing:01} 
Let $s$ be a positive integer, and let $t\in\{s-1,\,s\}$.
Let $X=\{x_1,x_2,\ldots,x_s\}$ and $Y=\{y_1,y_2,\ldots,y_t\}$ be ordered sets of real numbers.  We say that the set \emph{$X$ interlaces the set $Y$ from both sides}, denoted $X\bowtie Y$, if $t=s-1$ and
\begin{equation}\label{def:bowtie}
x_1<y_1<x_2<y_2<\cdots<x_{s-1}<y_{t}<x_s.
\end{equation}
Note that the bow-tie symbol $\bowtie$ consists of a ``times'' symbol $\times$ in the middle and a bar at each side.  The left (resp., right) bar indicates that the smallest (resp., largest) number in the chain~(\ref{def:bowtie}) is from the set~$X$. We say that the set \emph{$X$ interlaces $Y$ from the right}, denoted $X\rtimes Y$, if either~$X\bowtie Y$ or
\begin{equation}\label{def:rtimes}
t=s
\rmand
y_1<x_1<y_2<x_2<\cdots<x_{s-1}<y_t<x_s.
\end{equation}
Here the bar to the right of the ``times'' symbol~$\times$ within the symbol~$\rtimes$ means  that the largest number in the chain~(\ref{def:rtimes}) is from~$X$.  We observe that any set consisting of a single real number interlaces the empty set.
\end{definition}
\smallskip

For any integers $m\le n$, we denote the set $\{m,\,m+1,\ldots,\,n\}$ by $[m,n]$.  Moreover, when $m=1$, we may denote the set $[1,n]$ by~$[n]$.  Lemma~\ref{lem:fg:01} presents some essential consequences of the interlacing property.
\smallskip

\begin{lem}\label{lem:fg:01}
Let $f(x)$ and $g(x)$ be polynomials with zero-sets $X$ and $Y$ respectively.
Let $\beta\in\R$, and let 
\[ X'=X\cap(-\infty,\,\beta)=\{x_1,\,x_2,\,\ldots,\,x_p\} \rmand Y'=Y\cap(-\infty,\,\beta)=\{y_1,\,y_2,\,\ldots,\,y_q\}\]
be two ordered sets such that $X'\rtimes Y'$.
Let $x_0=y_0=-\infty$ and $y_{q+1}=\beta$.
\begin{itemize}
\item If $f(\beta)\ne0$, then we have
\begin{equation}\label{ineq:lem:f}
f(y_{j})f(\beta)(-1)^{q-j}<0\all{j\in[q+1-p,\,q+1]};
\end{equation}
\item If $g(\beta)\ne0$, then we have
\begin{equation}\label{ineq:lem:g}
g(x_{i})g(\beta)(-1)^{p-i}>0\all{i\in[p-q,\,p]}.
\end{equation}
\end{itemize}
\end{lem}
\begin{proof}
See Appendix \ref{prooflem:fg:01}.
\end{proof}
\smallskip

\begin{definition}
For any sequence $\{x_n\}$ of real numbers, we write $x_n\searrow x$ if $x_n$ converges to the number $x$ decreasingly, and we write $x_n\nearrow x$ if $x_n$ converges to the number $x$ increasingly.
\end{definition}

Our main result, Theorem \ref{Mthm:main:01}, concerns a polynomial sequence $\{W_n(x)\}$ of type $(0,1)$ in which $A(x)=a$ and $B(x)=bx+c$, with $ab\ne0$ and $c\in\R$. The proof of Theorem \ref{Mthm:main:01} is given in Section 3. 
\smallskip

\begin{thm}\label{Mthm:main:01}
Let $\{W_n(x)\}_{n\ge0}$ be the polynomial sequence defined by the recursion 
\begin{align}
W_n(x) ~=~ aW_{n-1}(x)+(bx+c)W_{n-2}(x)\label{Mrec:W:01}
\end{align}
with initial conditions $W_0(x)=1$ and $W_1(x)=t(x-r)$, where $a,b,t>0$, $c,r\in\R$, and $r\ne -c/b$.  Also, let $R_n=\{\xi_{n,1},\ldots,\xi_{n,d_n}\}$ be the ordered zero set of $W_n(x)$.
Then the polynomial $W_n(x)$ has degree $d_n= \lfloor(n+1)/2\rfloor$ and is distinct-real-rooted.  Moreover, using the notations 
$$x^*=-{4c+a^2\over 4b},\quad r^*=x^*-{a\over 2t}, \rmand y^*=r+{(at+b)-\sqrt{(at+b)^2+4t^2(br+c)}\over 2t^2}$$
we have the following conclusions: 
\begin{itemize}
\item[(i)]
If $r\in(-\infty,r^*]$,
then $R_{n+1}\rtimes R_n$
and $R_{n+2}\bowtie R_n$ for $n\ge1$;
$\xi_{n,\,d_n-i}\nearrow x^*$ for any fixed $i\ge0$.

\item[(ii)]
If $r\in(r^*,-c/b)$ then $R_{n+1}\rtimes R_n$
and $R_{n+2}\bowtie R_n$ for $n\ge1$;
$\xi_{n,\,d_n-i}\nearrow x^*$ for any fixed $i\ge1$; and $\xi_{n,d_n}\nearrow y^*$ with $x^*<y^*$.

\item[(iii)]
If $r\in(-c/b,+\infty)$ then
$R_{n+1}'\rtimes R_n'$
and $R_{n+2}'\bowtie R_n'$ for $n\ge3$;
$\xi_{n,d_n-i}\nearrow x^*$ for any fixed  $i\ge1$;
$\xi_{2n,d_{2n}}\nearrow y^*$ and $\xi_{2n-1,d_{2n-1}}\searrow y^*$ with $x^*<-c/b<y^*<x_{2,d_2}$.
\end{itemize}
The best bounds for the set $\cup_{n\ge1}R_n$ are, in these three respective cases,
$(-\infty,\,x^*)$, $(-\infty,\,y^*)$ and $(-\infty,\,r)$.  Furthermore, the sequence $\xi_{n,i}$ converges to $-\infty$ for any fixed $i\ge1$.  
\end{thm}
\smallskip

We observe that in the statement of Theorem \ref{Mthm:main:01}, the limit point $x^*$ does not depend on the initial polynomial $W_1(x)$, as long as the polynomial $W_1(x)$ is linear, and furthermore, no root lies in the interval $(x^*,\,-c/b)$ for case (iii).
\smallskip

\begin{definition}\label{def:W:01}
Let $\{W_n(x)\}=\{W_n(x)\}_{n\ge0}$ be the polynomial sequence defined recursively by
\begin{equation}\label{rec:01}
W_n(x)=aW_{n-1}(x)+(bx+c)W_{n-2}(x),
\end{equation}
with $W_0(x)=1$ and $W_1(x)=x$,
where $a,b>0$ and $c\neq0$.
In this context, we say $\{W_n(x)\}$ is a {\em $(0,1)$-sequence of polynomials}.
\end{definition}
\smallskip

\begin{thm}\label{Mthm:main:02}
Let $\{W_n(x)\}_{n\ge0}$ be a $(0,1)$-sequence of polynomials. Then the polynomial $W_n(x)$ (of degree $d_n= \lfloor(n+1)/2\rfloor$) is distinct-real-rooted.  Let
\begin{equation} \label{eq:x*r*y*}
x^*=-{4c+a^2\over 4b},\quad  r^*=x^*-{a\over 2}, \rmand y^*={a+b-\sqrt{(a+b)^2+4c}\over 2}.
\end{equation}
Let $R_n=\{\xi_{n,1},\ldots,\xi_{n,d_n}\}$ be the ordered zero-set of $W_n(x)$.  
\begin{itemize}
\item[(i)]
If $r^*\geq0$,
then $R_{n+1}\rtimes R_n$
and $R_{n+2}\bowtie R_n$ for $n\ge1$;
$\xi_{n,\,d_n-i}\nearrow x^*$ for any fixed $i\ge0$.

\item[(ii)]
If $0\in(r^*,-c/b)$ then $R_{n+1}\rtimes R_n$
and $R_{n+2}\bowtie R_n$ for $n\ge1$;
$\xi_{n,\,d_n-i}\nearrow x^*$ for any fixed $i\ge1$; and $\xi_{n,d_n}\nearrow y^*$ with $x^*<y^*$.

\item[(iii)]
If $c>0$ then
$R_{n+1}'\rtimes R_n'$
and $R_{n+2}'\bowtie R_n'$ for $n\ge3$;
$\xi_{n,d_n-i}\nearrow x^*$ for any fixed  $i\ge1$;
$\xi_{2n,d_{2n}}\nearrow y^*$ and $\xi_{2n-1,d_{2n-1}}\searrow y^*$ with $x^*<-c/b<y^*<x_{2,d_2}$.
\end{itemize}
For these three cases, the respective best bounds for the set $\cup_{n\ge1}R_n$ are $(-\infty,\,x^*)$, $(-\infty,\,y^*)$, and $(-\infty,\,r)$.  Moreover, the sequence $\xi_{n,i}$ converges to $-\infty$ for any fixed $i\ge1$.
\end{thm}

Now we explain how Theorem \ref{Mthm:main:01} can be obtained as a corollary of Theorem  \ref{Mthm:main:02}.
When considering the root geometry problem of general recursive polynomial sequences of type~$(0,1)$, it is acceptable to assume that $\deg W_0(x)\leq \deg W_1(x)$ and that the polynomial $W_0(x)$ is monic, which implies that $W_0(x)=1$.
Note that if $a<0$, then  when considering the polynomial sequence
\begin{equation}\label{-Tr}
\tW_n(x)=(-1)^nW_n(-x).
\end{equation}
it is routine to verify that $\tW_0(x)=1$, $\tW_1(x)=x$,
and $\tW_n(x)=-a\tW_{n-1}(x)+(-bx+c)\tW_{n-2}(x)$.
Consequently, supposing that $a>0$ is without loss of generality in regard to the root geometry.

We now explain why it is enough to prove Theorem \ref{Mthm:main:02} for only the case in which 
\begin{equation}
(i)\, W_1(x)=x, \qquad
(ii)\, c\ne0,\qquad
(iii)\, b>0.
\end{equation}
\begin{itemize}
\item[(i)] Here the linear polynomial $W_1(x)$ can be supposed to have the form $t(x-r)$.
We can always normalize the polynomials by the linear transformation
\begin{equation}\label{LinearTr}
\tW_n(x)\=W_n(x/t+r),
\end{equation}
whose root geometry differs from that of the sequence $W_n(x)$ only by magnification and translation.

\item[(ii)] If $c=0$, then the number $0$ is a root of every polynomial $W_n(x)$.
In this circumstance, one may consider the polynomials $\tW_n(x)$ defined by the rule
$$\tW_n(x)\={W_{n+2}(x)\over W_2(x)}.$$
It is clear that $\{\tW_n(x)\}_{n\ge0}$ satisfies the recursion $\tW_n(x)=a\tW_{n-1}(x)+(bx+c)\tW_{n-2}(x)$.
Therefore, the condition $c\neq 0$ is not really restrictive.

\item[(iii)]  The case in which $b<0$ is unexplored.
In fact, when $b<0$,
both the degrees and the leading coefficients of the polynomials $W_n(x)$ may vary irregularly.
We also note that dropping Condition~(iii)  may yield non-real-rooted polynomials $W_n(x)$.
For example, when $a=1$, $b=-1$, $c=-1$,
we have $W_3(x)=-x^2-x-1$, which has no real roots.
\end{itemize}

We remark that in a general setting, beyond the genus polynomials of graphs, the polynomials $W_n(x)$ might have negative coefficients. In summary, this study of the root geometry of recursive polynomials of type $(0,1)$ has only two restrictions. One is that the polynomial $W_0(x)$ is a constant. The other is the assumption that the number $b$ is positive.
\medskip

\subsection{Some examples}

We now present several examples to illustrate our results.

\begin{eg}\label{eg:Fib}
One kind of sequence of Fibonacci polynomials $W_n(x)$ is defined by the recursion
\begin{equation}\label{rec:Fib}
W_n(x)\=W_{n-1}(x)+xW_{n-2}(x),
\end{equation}
where $W_0(x)=1$ and $W_1(x)=x+1$; see~\cite[Table 3]{MSV02} and~\cite[A011973]{OEIS}.  Accordingly, $a=b=1$, $c=0$, and $r=-1$, and we compute from Equation \eqref{eq:x*r*y*} that 
$$x^*=-{4c+a^2\over 4b} = -\frac{1}{4} \rmand  r^*=x^*-{a\over 2} = -\frac{1}{4}-\frac{1}{2} =-\frac{3}{4} > -1 = r .$$
By Theorem~\ref{Mthm:main:01}(i), we know that each polynomial $W_n(x)$ is distinct-real-rooted and that all roots are less than $-1/4$.  Also, for any $\epsilon>0$, there exists a number $M'>0$ such that every polynomial $W_n(x)$ with $n>M'$ has a root in the interval $(-1/4-\epsilon,\,-1/4)$.  Moreover, by the final conclusion of Theorem~\ref{Mthm:main:01}, we know that for any $N>0$, there exists a number $M>0$ such that every polynomial $W_n(x)$ with $n>M$ has a root less than $-N$. 
\end{eg}

In the next two examples, we examine how the set of convergent points is affected when we change the coefficient of  $W_{n-2}(x)$ in Recursion~(\ref{rec:Fib}) first to $2x/5$ and then to $x+2$. 
\begin{eg}
Let $W_n(x)$ be the polynomial sequence defined by the recursion 
$$W_n(x)\=W_{n-1}(x)+{2x\over 5}W_{n-2}(x),$$ 
with initial values $W_0(x)=1$ and $W_1(x)=x+1$.  We see that  $a=1$, $b=2/5$, $c=0$, and $r=-1$.  We calculate from Equation \eqref{eq:x*r*y*} that 
$$x^*=-{4c+a^2\over 4b}=-\frac{5}{8}<-\frac{3}{5}=y^* \rmand r=-1\in\left(-\frac{9}{8},\,0\right)=\left(r^*,\, -\frac{c}{b}\right).$$  
By Theorem~\ref{Mthm:main:01}, the polynomial $W_n(x)$ is distinct-real-rooted, and the largest root converges to $-3/5$ increasingly. Moreover, for any positive integer $i$, the root sequence $x_{n,\,d_n-i}$ converges to $-5/8$ increasingly, and the root sequence $x_{n,\,i}$ converges to $-\infty$ decreasingly.
\end{eg}

\begin{eg}
Let $W_n(x)$ be the polynomial sequence defined by the recursion  
$$W_n(x)\=W_{n-1}(x)+(x+2)W_{n-2}(x),$$ 
with initial values  $W_0(x)=1$ and $W_1(x)=x+1$.  Thus, $a=b=1$, $c=2$, and $r=-1$.  We compute that $W_2(x)=2x+3$, and that 
$$x^*=-\frac{9}{4},\qquad r=-1>-2=-\frac{c}{b}, \rmand y^*=-\sqrt{2}.$$  
Therefore, we have $x^*<-c/b<x_{2,\,d_2}$. By Theorem~\ref{Mthm:main:01}, each of the polynomials~$W_n(x)$ is distinct-real-rooted, and has exactly one root larger than $-9/4$. The sequence of largest roots converges to $-\sqrt{2}$ oscillatingly.  Moreover, for any positive integer~$i$, the root sequence $x_{n,\,d_n-i}$ converges to $-9/4$ increasingly, and the root sequence $x_{n,\,i}$ converges to $-\infty$ decreasingly.
\end{eg}

\begin{eg}
This example illustrates how our results can be used to prove the real-rootedness of a sequence of partial genus polynomials. Let $D_n(x)$ be the polynomial sequence defined by the recursion $D_n(x) =2D_{n-1}(x)+8xD_{n-2}(x)$, with $D_0(x)=1$ and $D_1(x)=2x$, which may be recognized by those familiar with enumerative research in topological graph theory (for example, see \cite{FGS89,GF87,GRT89}) as a partial genus distribution for the closed-end ladder $L_n$, which is shown in Figure \ref{fig:L4}.
\begin{figure} [ht]
\centering
    \includegraphics[width=1.8in]{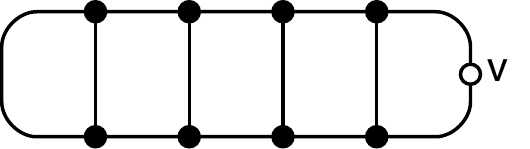} \vskip-6pt
\caption{The closed-end ladder $L_4$ with a 2-valent root-vertex $v$.}
\label{fig:L4}
\end{figure}

\noindent The polynomial $D_n(x)$ is the generating function for the number of cellular imbeddings of the ladder $L_n$ such that two different faces are incident on the root-vertex.  By Theorem \ref{Mthm:main:01}, each $D_n(x)$ is a distinct-real-rooted polynomial, and the root sequence $\xi_{n,\,d_n-i}$ converges to $-1/8$ for every nonnegative integer $i$. In particular, none of the polynomials $D_n(x)$ has a root larger than $-1/8$. Unfortunately, we do not yet know what topological information is implied by this convergent point.
\end{eg}

\bigskip  
%%%%%%%%%%%%%%%%%%%%%%%%%%%%%%%%%%%%%%%%%%%%%
%%%%%%%%%%%%%%%%%%%%%%%%%%%%%%%%%%%%%%%%%%%%%
\section{\large Distinct Real-Rootedness \label{thm:rr}}

The proof of Theorem \ref{Mthm:main:02} begins here with an investigation of the real-rootedness of a $(0,1)$-sequence of polynomials.  The remainder of the proof will be given in Section \ref{sec:bzs} and Section \ref{sec:limpts}. 

For any polynomial $f(x)$, we follow the usual definition that $f(\pm\infty)=\lim_{x\to\pm\infty}f(x)$.  We start our analysis of $(0,1)$-sequences  $\{W_n(x)\}_{n\ge0}$ by finding a formula for the degree and the leading coefficient of each of the polynomials.
\smallskip

\begin{lem} \label{lem:deg:01}
Let $\{W_n(x)\}_{n\ge0}$ be a $(0,1)$-sequence of polynomials, with the constant $b$ as in Definition~\ref{def:W:01}, and with $t_n$ the leading coefficient of $W_n(x)$.   Then
$$d_n=deg(W_n(x))=\dfl{\frac{n+1}{2}},\quad t_{2n+1}=b^{n}, \rmand t_{2n}=b^{n-1}(na+b).$$
Moreover, for all ${n\ge1}$, we have
\[
W_n(-\infty)(-1)^{d_n}>0
\rmand
W_n(+\infty)>0
\]
\end{lem}
\begin{proof}
The formulas for the degree~$d_n$ and the leading coefficients $t_n$ can be verified by induction on the integer $n$.
For any polynomial $f(x)$ with positive leading coefficient, it is clear that
\[
f(-\infty)(-1)^{\deg f(x)}=+\infty
\rmand
f(+\infty)=+\infty.
\]
Since $t_n>0$, we infer that
$W_n(-\infty)(-1)^{d_n}=W_n(+\infty)=+\infty$.
The sign relations follow immediately.
\end{proof}
\smallskip

Using the intermediate value theorem for a $(0,1)$-sequence of polynomials, we derive the following criterion for their distinct-real-rootedness.
\smallskip

\begin{thm}\label{thm:criterion:interlacing:01}
Let $\{W_n(x)\}_{n\ge0}$ be a $(0,1)$-sequence of polynomials.  Let $d_n=deg(W_n(x))$,   and let $\beta\le -c/b$.  We denote the ordered zero-set of the polynomial $W_n(x)$ by $R_n$.  Suppose that for some numbers $m,k\in\mathbb{N}$, we define 
\begin{equation} \label{eq:defTm}
T_m = R_m\cap(-\infty,\,\beta) \rmand T_{m+1} = R_{m+1}\cap(-\infty,\,\beta),
\end{equation}
and suppose, further, that 
\begin{align}
\label{cond:RR:beta:01}
&W_m(\beta)(-1)^{k}\>0,\\
\label{cond:RR:m:01}
&|T_m|\=d_m-k,\\
\label{cond:RR:m+1:01}
&|T_{m+1}|\=d_{m+1}-k, \rmand\\
\label{cond:RR:rtimes}
&T_{m+1}\;\rtimes\;T_{m}.
\end{align}
Then there exists a set $T_{m+2}\subseteq R_{m+2}\cap (-\infty,\,\beta)$ such that $|T_{m+2}|=d_{m+2}-k$ and, furthermore, such that $T_{m+2}\rtimes T_{m+1}$.  Moreover, if
\begin{align}
\label{cond:RR:m+2:01}
T_{m+2}&\=R_{m+2}\cap (-\infty,\,\beta),
\end{align}
then we have $T_{m+2}\bowtie T_m$.
\end{thm}

\begin{proof}
By Conditions~(\ref{cond:RR:m:01}) and~(\ref{cond:RR:m+1:01}) of the premises, we can suppose that
\[
T_{m+1}\=\{x_1,\,x_2,\,\ldots,\,x_p\}
\rmand
T_m\=\{y_1,\,y_2,\,\ldots,\,y_q\}
\]
are ordered sets, where $p=d_{m+1}-k$ and $q=d_m-k$.  Definition~(\ref{eq:defTm}) implies that $x_p<\beta$.  In view of Condition~(\ref{cond:RR:rtimes}), together with the premise $\beta\leq-c/b$, we have the following ordering:
\begin{equation}\label{pf:ordering:01}
\cdots\<y_{q-2}\<x_{p-2}\<y_{q-1}\<x_{p-1}\<y_q\<x_p\<\beta\;\le\;-c/b.
\end{equation}
Note that Condition~(\ref{cond:RR:rtimes}) also implies that $p\ge1$ and $q\in\{p-1,\,p\}$.  For convenience, let
\[x_0=y_0=-\infty \rmand x_{p+1}=y_{q+1}=\beta.\]
By applying Lemma~\ref{lem:fg:01} with $f(x)=W_{m+1}(x)$ and $g(x)=W_m(x)$, we obtain that the zero-sets~$X$ and~$Y$ become~$R_{m+1}$ and~$R_m$ respectively. Consequently, by Definition~(\ref{eq:defTm}), we have
\[
X'=R_{m+1}\cap(-\infty,\,\beta)=T_{m+1}
\rmand
Y'=R_{m}\cap(-\infty,\,\beta)=T_{m}.
\]
From Inequality~(\ref{cond:RR:beta:01}) in the premises, we infer that
\begin{equation}\label{pf:ineq:61}
W_m(\beta)\ne0.
\end{equation}
Therefore, we can use Inequality~(\ref{ineq:lem:g}), which gives that
\begin{equation}\label{pf:ineq1}
W_m(x_{i})W_m(\beta)(-1)^{p-i}>0\all{i\in[p-q,\,p]}.
\end{equation}

Let $i\in[p]$.
Since $x_i\in T_{m+1}\subseteq R_{m+1}$, we have $W_{m+1}(x_i)=0$.
Taking $n=m+2$ and $x=x_i$ in Recursion (\ref{Mrec:W:01}), we see that
$W_{m+2}(x_{i})\=(bx_{i}+c)W_{m}(x_{i})$.
From (\ref{pf:ordering:01}), we see that $x_i<-c/b$.
Since $b>0$, we deduce that $bx_i+c<0$.
Thus, we can substitute $W_m(x_{i})=W_{m+2}(x_{i})/(bx_{i}+c)$ into Inequality~(\ref{pf:ineq1}), which gives that
\[
\frac{W_{m+2}(x_{i})}{bx_{i}+c}W_{m}(\beta)(-1)^{p-i}>0.
\]
Since $bx_i+c<0$, the above inequality can be reduced to
\begin{equation}\label{pf:ineq2}
W_{m+2}(x_{i})W_{m}(\beta)(-1)^{p-i}\<0.
\end{equation}
Note that Inequality~(\ref{pf:ineq2}) holds also for $i=p+1$.
Replacing $i$ by $i+1$ in Inequality~(\ref{pf:ineq2}) gives that
$W_{m+2}(x_{i+1})W_{m}(\beta)(-1)^{p-i}\>0$.
Multiplying it by Inequality~(\ref{pf:ineq2}), we obtain that
\[
W_{m+2}(x_i)W_{m+2}(x_{i+1})\<0.
\]
By the intermediate value theorem,
the polynomial $W_{m+2}(x)$ has a root in the interval $(x_{i},\,x_{i+1})$.
Let~$z_{i}$ be such a root.

When $i=1$, Inequality~(\ref{pf:ineq2}) is $W_{m+2}(x_{1})W_{m}(\beta)(-1)^{p-1}\<0$.
Multiplying it by Inequality~(\ref{cond:RR:beta:01}) gives that
$W_{m+2}(x_1)(-1)^{p-1+k}<0$.
Since $p=d_{m+1}-k$, the above inequality is $W_{m+2}(x_1)(-1)^{d_{m+1}}>0$.
On the other hand, Lemma~\ref{lem:deg:01} gives that $W_{m+2}(-\infty)(-1)^{d_{m+2}}>0$.
Since $d_{m+1}+d_{m+2}=m+2$, we obtain that
\[
W_{m+2}(-\infty)W_{m+2}(x_1)(-1)^{m+2}>0.
\]
Therefore, by the intermediate value theorem,
the polynomial $W_{m+2}(x)$ has a root in the interval $(-\infty,\,x_1)$ when $m$ is odd.
Let~$z_0$ be such a root. 

Define
\begin{equation}\label{def:Tm+2}
T_{m+2}=\begin{cases}
\{z_1,\,z_2,\,\ldots,\,z_p\},&\textrm{if $m$ is even};\\[3pt]
\{z_1,\,z_2,\,\ldots,\,z_p\}\cup\{z_{0}\},&\textrm{if $m$ is odd}.
\end{cases}
\end{equation}
We shall now show that this set $T_{m+2}$ has the desired properties.
\begin{itemize}
\smallskip
\item
For each $j\in[0,p]$, the number~$z_j$ is chosen to be a zero of the polynomial $W_{m+2}(x)$. Therefore, 
$T_{m+2}\subseteq R_{m+2}$.
\smallskip\item
For each $j\in[0,p]$, the number $z_j$ is chosen from the interval $(x_j,\,x_{j+1})$, which is contained in the interval $(-\infty,\,\beta)$. Therefore, 
$T_{m+2}\subset (-\infty,\,\beta)$.
\smallskip\item
From Definition~(\ref{def:Tm+2}), we see that
\begin{itemize}
\smallskip\item
if $m$ is even, then $|T_{m+2}|=p=d_{m+1}-k=(m+2)/2-k=d_{m+2}-k$;
\smallskip\item
if $m$ is odd, then $|T_{m+2}|=p+1=d_{m+1}-k+1=(m+3)/2-k=d_{m+2}-k$.
\end{itemize}
Hence, in any case, we have that $|T_{m+2}|=d_{m+2}-k$.
\smallskip\item
Since $z_j\in(x_j,\,x_{j+1})$ for all $j\in[0,p]$, we have $T_{m+2}\rtimes T_{m+1}$
according to Definition~\ref{def:interlacing:01}.
\end{itemize}

It remains to show that $T_{m+2}\bowtie T_m$.
By applying Lemma~\ref{lem:fg:01} with $f(x)=W_{m+2}(x)$ and with $g(x)=W_m(x)$, we obtain that
the zero-sets~$X$ and~$Y$ become~$R_{m+2}$ and~$R_m$ respectively.
Consequently, by Conditions~(\ref{cond:RR:m+2:01}) and~(\ref{cond:RR:m:01}) of the premises, we have
\[
X'=R_{m+2}\cap(-\infty,\,\beta)=T_{m+2}
\rmand
Y'=R_{m}\cap(-\infty,\,\beta)=T_{m}.
\]
Since $d_{m+2}=d_m+1$, the result $|T_{m+2}|=d_{m+2}-k$ and Condition~(\ref{cond:RR:m:01}) imply that
\[
|X'|=|T_{m+2}|=d_{m+2}-k=d_m+1-k=q+1=|T_m|+1=|Y'|+1.
\]
Therefore, the lower bound $|Y'|-|X'|+1$ of the range of the index~$j$ in Inequality~(\ref{ineq:lem:f}) is~$0$.
With the aid of Inequality~(\ref{pf:ineq:61}), we can use Inequality~(\ref{ineq:lem:f}), which gives that
\[
W_{m+2}(y_{j})W_{m+2}(\beta)(-1)^{q-j}<0\all{j\in[0,\,q+1]}.
\]
Let $j\in[q+1]$. Replacing $j$ by $j-1$ in the above inequality, we obtain that
\[
W_{m+2}(y_{j-1})W_{m+2}(\beta)(-1)^{q-j}>0.
\]
Multiplying the above two inequalities gives that
\begin{align}\label{dsr:bowtie1a2}
W_{m+2}(y_{j-1})W_{m+2}(y_{j})<0\all{j\in[q+1]}.
\end{align}
By the intermediate value theorem, we infer that the polynomial $W_{m+2}(x)$ has a root, say,~$w_j$,
in the interval $(y_{j-1},\,y_j)$, that is,
\[
-\infty\<w_1\<y_1\<w_2\<y_2\<\cdots\<w_q\<y_q\<w_{q+1}\<y_{q+1}=\beta.
\]
Define $T=\{w_1,\,w_2,\,\ldots,\,w_{q+1}\}$. Then the above chain of inequalities implies that $T\bowtie T_m$.
By the choice of the numbers $w_j$, we see that $w_j\in R_{m+2}$ and $w_j<\beta$.
It follows that $T\subseteq R_{m+2}\cap(-\infty,\,\beta)=T_{m+2}$.
Since $|T|=q+1=|T_{m+2}|$, we conclude that $T=T_{m+2}$.
Hence, the above interlacing relation $T\bowtie T_m$ becomes $T_{m+2}\bowtie T_m$, which completes the proof.
\end{proof}
\smallskip

The usage of this method of interlacing dates back at least to Harper~\cite{Har67}, who established the real-rootedness of the Bell polynomials in this way.  We should mention that Liu and Wang~\cite{LW07} have further applied this method to establish several easy-to-verify criteria for the real-rootedness of polynomials in which all the coefficients are non-negative.

We make two small preparations and a lemma before the further exploration of the root geometry,
which will be used in several proofs below. The polynomial $W_1(x)=x$ has the unique root $y_1=0$.
From Recursion~(\ref{Mrec:W:01}), it is direct to compute that the polynomial $W_2(x)=(a+b)x+c$,
which has the unique root $y_2=-c/(a+b)$. Therefore, the polynomials $W_1(x)$ and $W_2(x)$ are real-rooted, and
\begin{equation}\label{R1R2:01}
R_1=\{0\}
\rmand
R_2=\{-c/(a+b)\}.
\end{equation}
In the remainder of this section, we will use Theorem~\ref{thm:criterion:interlacing:01} frequently.  We will always set the constant~$k$ to be either $0$ or $1$.  The following corollary is for the particular case $k=0$.
\smallskip

\begin{lem}\label{lem:k=0:01}
Let $\{W_n(x)\}_{n\ge0}$ be a $(0,1)$-sequence of polynomials.
Denote the ordered zero-set of $W_n(x)$ by $R_n$.
Let $c<0$, and
\begin{equation}\label{cond:k=0:01}
-c/(a+b)<\beta\le -c/b.
\end{equation}
Let $N$ be a positive integer.
If $W_m(\beta)>0$ for all $m\in[N]$, then we have
\begin{align*}
R_m\subset (-\infty,\,\beta)&\all{m\in[N+2]},\\
R_{m+1}\rtimes R_m&\all{m\in[N+1]},\\
\text{and}\qquad R_{m+2}\bowtie R_m&\all{m\in [N]}.
\end{align*}
In particular, if $W_m(\beta)>0$ for all $m\ge 1$, then the above three relations hold for all $m\ge 1$.
\end{lem}

\begin{proof}
First, we show the following relations by induction on the integer~$m$:
\begin{equation}\label{pf1:lem:01}
R_m\subset (-\infty,\,\beta),\qquad
R_{m+1}\subset (-\infty,\,\beta),
\rmand
R_{m+1}\rtimes R_m
\all{m\in[N+1]}.
\end{equation}

Recall from Formula~(\ref{R1R2:01}) that $R_1=\{0\}$ and $R_{2}=\{-c/(a+b)\}$.
When $m=1$, the relations in~(\ref{pf1:lem:01}) become
\[
R_1\subset (-\infty,\,\beta),\qquad
R_{2}\subset (-\infty,\,\beta),
\rmand
R_{2}\rtimes R_1,
\]
that is,
\[
0<\beta,\qquad
-c/(a+b)<\beta,
\rmand
0<-c/(a+b).
\]
Since $a,b>0$, the above relations hold by the negativity of the number~$c$ and Inequality~(\ref{cond:k=0:01}) in the premises.
Suppose that the relations in~(\ref{pf1:lem:01}) hold for $m\in [N]$, and we need to show them for $m+1$.

Let $k=0$. The upper bound $-c/b$ of the parameter $\beta$ is as same as that in Theorem~\ref{thm:criterion:interlacing:01}.
We are going to verify Conditions~(\ref{cond:RR:beta:01})--(\ref{cond:RR:rtimes}).
Since $k=0$, the inequality $W_m(\beta)>0$ in the premises is exactly Conditions~(\ref{cond:RR:beta:01}).
From Definition~(\ref{eq:defTm}) and the induction hypothesis $R_m\subset (-\infty,\,\beta)$, we infer that
$T_m=R_m\cap (-\infty,\,\beta)=R_m$.
It follows that $|T_m|=|R_m|=d_m$, i.e., Condition~(\ref{cond:RR:m:01}) holds.
Similarly, we have $T_{m+1}=R_{m+1}$, i.e., Condition~(\ref{cond:RR:m+1:01}) holds.
Thus, by the induction hypothesis we have that $R_{m+1}\rtimes R_m$, which is equivalent to Condition~(\ref{cond:RR:rtimes}).

Therefore, we can apply Theorem~\ref{thm:criterion:interlacing:01} and obtain the existence of
a set $T_{m+2}\subseteq R_{m+2}\cap(-\infty,\,\beta)$ such that $T_{m+2}\rtimes T_{m+1}$ and $|T_{m+2}|=d_{m+2}$.
Since the sets~$T_{m+2}$ and~$R_{m+2}$ have the same cardinality~$d_{m+2}$, we obtain that
\begin{equation}\label{pf65}
T_{m+2}=R_{m+2}.
\end{equation}
Consequently,
the result $T_{m+2}\subset (-\infty,\,\beta)$ becomes the desired relation
\begin{equation}\label{pf66}
R_{m+2}\subset (-\infty,\,\beta);
\end{equation}
and the result $T_{m+2}\rtimes T_{m+1}$ becomes the desired relation $R_{m+2}\rtimes R_{m+1}$.
This completes the induction proof for the relations in~(\ref{pf1:lem:01}).

By Equation~(\ref{pf65}) and Relation~(\ref{pf66}), we infer that $T_{m+2}=R_{m+2}\cap (-\infty,\,\beta)$,
which is exactly Condition~(\ref{cond:RR:m+2:01}).
Hence, by Theorem~\ref{thm:criterion:interlacing:01}, we derive that $T_{m+2}\bowtie T_m$,
i.e., $R_{m+2}\bowtie R_m$, which completes the proof.
\end{proof}

\eject

In order to continue with our discussions, we fix more parameters of $(0,1)$-sequences of polynomials. Inspired by Lemma \ref{lem:00}, we introduce the following notations. 
\smallskip

\begin{definition}\label{def:xg}
Let $\{W_n(x)\}_{n\ge0}$ be a $(0,1)$-sequence of polynomials. We define
\begin{align*}
\Delta(x)&\=a^2+4(bx+c)\=4bx+a^2+4c,\\[3pt]
g^\pm(x)&\=\bigl(2x-a\pm\sqrt{\Delta(x)}\,\bigr)/2\=\bigl(2x-a\pm\sqrt{4bx+a^2+4c}\,\bigr)/2,\\[5pt]
g(x)&\=g^-(x)g^+(x)\=x^2-(a+b)x-c.
\end{align*}
We denote the zeros of the functions $B(x)=bx+c$, $\Delta(x)$ and $g^+(x)$ by
\begin{equation}\label{def:xB:xd:xg}
x_B\=-\frac{c}{b},\quad
x_\Delta\=-{a^2+4c\over 4b}, \rmand
x_g\={(a+b)-\sqrt{(a+b)^2+4c}\over 2},
\end{equation}
respectively.  We also define
$$ n_0 = {2ab\over a^2+2ab+4c}. $$
We observe that Lemma~\ref{lem:00} implies the following:
\begin{align}
W_n(x_B)&\=a^{n-1}W_1(x_B), \label{W:xb}\\
W_n(x_\Delta)&\=\bgg{1+{n(2x_\Delta-a)\over a}}\bgg{{a\over 2}}^{\!\!n}, \mbox{ and } \label{W:xd}\\
W_n(x_g)&\=x_g^n.\label{W:xg}
\end{align}
\end{definition}
\smallskip

The following technical lemma provides the ordering among the numbers $x_\Delta$, $x_g$, $x_B$, and $0$, for the sake of determining the sign of the value $W_n(x)$ for specific numbers $x$ in the proofs of Theorem~\ref{thm:RR:01}.

\begin{lem}\label{lem:order:01}
Let $\{W_n(x)\}_{n\ge0}$ be a $(0,1)$-sequence of polynomials with parameters specified as in Definition \ref{def:xg}.
\begin{itemize}
\item[(i)]
If $\frac{4c}{a^2+2ab}\le-1$, then $W_n(x_\Delta)>0$ for all $n\ge1$.
\smallskip
\item[(ii)]
If $\frac{4c}{a^2+2ab}>-1$,then $x_g\in\R$, $x_\Delta<x_g$, and
\begin{equation}\label{eqrl:m0:01}
W_n(x_\Delta)\ge 0 \eqrl n\le n_0={2ab\over a^2+2ab+4c},
\end{equation}
where the equality on the left hand side holds if and only if the equality on the right hand side holds.
\noindent Moreover, if $c>0$, then $W_n(x_\Delta)<0$ and $x_B<x_g<0$;
otherwise, we have $0<x_g<x_B$.
\end{itemize}
\end{lem}
\begin{proof}
See Appendix \ref{prooflem:order:01}.
\end{proof}

Below is an example illustrating the cases $-(a^2+2ab)/4<c<0$ and $c>0$, respectively.

\begin{eg}\label{eg2:order}
Let $\{W_n(x)\}_{n\ge1}$ be a $(0,1)$-sequence of polynomials with parameters specified as in Definition \ref{def:xg}, and with  $a=b=1$ and $c=-1/2$; thus, we are in Case (ii), and  we have
$$W_n(x)=W_{n-1}(x)+(x-1/2)W_{n-2}(x),$$
with initial conditions $W_0(x)=1$ and $W_1(x)=x$. We observe that $-(a^2+2ab)/4<c<0$.  By Definition \ref{def:xg}, we have
$$x_\Delta = -\frac{a^2+4c}{4b} = -\frac{1-2}{4} = 1/4,$$
$x_g=1-\sqrt{2}/2$, $x_B=1/2$, and $n_0=2$.  Thus,  we may confirm the inequalities $x_\Delta<x_g$ and $0<x_g<x_B$. To illustrate that the sign of the value ${W_n(x_\Delta)}$ satisfies the left side of the equivalence relation~(\ref{eqrl:m0:01}), we calculate that $W_1(x_\Delta)=1/4>0$, $W_2(x_\Delta)=0$, and  $W_3(x_\Delta)=-1/16<0$. Continuing recursively, we see that $W_n(1/4)< 0$ for $n\ge3$.
\end{eg}
\smallskip

\begin{eg}\label{eg3:order}
Let $\{W_n(x)\}_{n\ge1}$ be a $(0,1)$-sequence of polynomials  with parameters specified as in Definition \ref{def:xg}, and with  $a=b=1$ and $c=1$, which puts us in Case (i); here we have $$W_n(x)=W_{n-1}(x)+(x+1)W_{n-2}(x),$$
with $W_0(x)=1$ and $W_1(x)=x$. This time, we have $c>0$. By Definition \ref{def:xg}, we have $x_\Delta=-5/4$, $x_g=1-\sqrt{2}$, $x_B=-1$, and $n_0=2/7$. These data correspond to the inequalities $x_B<x_g$ and $x_\Delta<x_g$ in the conclusion of Lemma \ref{lem:order:01}.  In fact, when $c>0$, we always have $n_0<1$, which implies that $W_n(x_\Delta)<0$.
\end{eg}
\smallskip

We are now ready to establish the real-rootedness of every polynomial $W_n(x)$.

\begin{thm}\label{thm:RR:01}
Let $\{W_n(x)\}_{n\ge1}$ be a $(0,1)$-sequence of polynomials with parameters specified as in Definition \ref{def:xg}. Then every polynomial $W_n(x)$ is distinct-real-rooted.  Moreover, let us denote the ordered zero-set of $W_n(x)$ by $R_n$, and let $y_n=\max R_n$ be the largest real root of the polynomial~$W_n(x)$. For all $n\ge 1$, we may conclude the following:
\begin{itemize}
\item[(i)]
if $c<0$, then $y_n<x_B$, $R_{n+1}\rtimes R_{n}$, and $\mR_{n+2}\bowtie R_{n}$.
\smallskip
\item[(ii)]
if $c>0$, then $y_n>x_B$, $R'_{n+1}\subset (-\infty,\,x_\Delta)$, $R'_{n+2}\rtimes \mR'_{n+1}$, and $R'_{n+2}\bowtie R'_{n}$, where $R'_n=R_n\setminus\{y_n\}$.
\end{itemize}
\end{thm}
\begin{proof}
From Equation~(\ref{W:xb}), we see that
\begin{equation}\label{pf371}
cW_n(x_B)<0.
\end{equation}
Below we will show (i) and (ii) individually.

\medskip

\noindent{\bf (i)} Let $c<0$. Then Inequality~(\ref{pf371}) reduces to $W_n(x_B)>0$ for all $n\ge 1$.
Take $\beta=-c/b$. Then Condition~(\ref{cond:k=0:01}) holds trivially.
By Lemma~\ref{lem:k=0:01}, we deduce that $R_{n}\subset (-\infty,\,x_B)$,
 $R_{n+1}\rtimes R_{n}$ and $R_{n+2}\bowtie R_n$ for all $n\ge1$.

\medskip

\noindent{\bf (ii)} Let $c>0$. Then Inequality~(\ref{pf371}) implies that $W_n(x_B)<0$.
By Lemma \ref{lem:deg:01}, we have $W_n(+\infty)>0$.
Therefore, by the intermediate value theorem,
the polynomial $W_n(x)$ has a real root in this interval $(x_B,+\infty)$.
In particular, the largest root~$y_n$ is larger than~$x_B$.
Note that $x_\Delta=-(a^2+4c)/(4b)<-c/b=x_B$. Thus, we have
\begin{equation}\label{pf5}
x_\Delta<x_B<y_n\all{n\ge1}.
\end{equation}
For the remaining desired relations, it suffices to show the following:
\begin{equation}\label{dsr:thm39}
R_n'\subset(-\infty,\,x_\Delta),\quad
R_{n+1}'\subset(-\infty,\,x_\Delta),\quad
R_{n+1}'\rtimes R_n',\quad
R_{n+1}'\bowtie R_{n-1}',\all{n\ge2}.
\end{equation}

We proceed by induction on $n$. Consider $n=2$. Since $d_1=d_2=1$, we have $R_1'=R_2'=\emptyset$.
Since $a,b,c>0$, from Definition~(\ref{def:xB:xd:xg}), we have
\[
x_\Delta=-(a^2+4c)/(4b)<0.
\]
In view of Formula~(\ref{W:xd}), we deduce that
\begin{equation}\label{pf4}
W_n(x_\Delta)\=\bgg{1+{n(2x_\Delta-a)\over a}}\bgg{{a\over 2}}^n<0,\all{n\ge1}.
\end{equation}
In particular, we have $W_3(x_\Delta)<0$.
On the other hand, Lemma~\ref{lem:deg:01} gives that $W_3(-\infty)(-1)^{d_3}>0$.
Since $d_3=2$, it reduces to $W_3(-\infty)>0$.
Therefore, by the intermediate value theorem,
we infer that the polynomial $W_3(x)$ has a root, say,~$r_3$, in the interval $(-\infty,\,x_\Delta)$.
From Inequality~(\ref{pf5}), we see that $y_3>x_\Delta$, and thus, $r_3<x_\Delta<y_3$.
It follows that $R_3=\{r_3,\,y_3\}$, and thus, $R_3'=\{r_3\}$.
Therefore, the relations in~(\ref{dsr:thm39}) for $n=2$ are respectively
\[
\emptyset\subset(-\infty,\,x_\Delta),\qquad
\{r_3\}\subset(-\infty,\,x_\Delta),\qquad
\{r_3\}\rtimes \emptyset,\qquad
\{r_3\}\bowtie \emptyset,
\]
all of which hold trivially, except the second one holds since $r_3<x_\Delta$.
\smallskip

Suppose that the $4$ relations in~(\ref{dsr:thm39}) hold for some $n\ge2$, by induction, it suffices to show that
\begin{equation}\label{dsr:thm391}
R_{n+2}'\subset(-\infty,\,x_\Delta),\qquad
R_{n+2}'\rtimes R_{n+1}',
\rmand
R_{n+2}'\bowtie R_{n}'.
\end{equation}
In applying Theorem~\ref{thm:criterion:interlacing:01},
we set $k=1$, $\beta=x_\Delta$ and $m=n$. We shall verify Conditions~(\ref{cond:RR:beta:01})--(\ref{cond:RR:rtimes}).

\begin{itemize}
\smallskip\item
Inequality~(\ref{pf4}) checks the truth for Condition~(\ref{cond:RR:beta:01}).
\smallskip
\item
From Definition~(\ref{eq:defTm}), we have $T_n=R_n\cap(-\infty,\,x_\Delta)$.
Note that in the zero-set $R_n$, except the largest root $y_n$, which is not in the interval $(-\infty,\,x_\Delta)$ by~(\ref{pf5}),
all the other roots (whose union is the set~$R_n'$) are in the interval $(-\infty,\,x_\Delta)$ by Hypothesis~(\ref{dsr:thm39}).
Therefore, we infer that $R_n\cap(-\infty,\,x_\Delta)=R_n'$, and thus, $T_n=R_n'$.
It follows that $|T_n|=|R_n'|=d_n-1$, which verifies Condition~(\ref{cond:RR:m:01}).
\smallskip
\item
Similarly, we have $T_{n+1}=R_{n+1}'$,
and Condition~(\ref{cond:RR:m+1:01}) holds true.
\item
Consequently, the hypothesis $T_{n+1}\rtimes T_n$ in~(\ref{dsr:thm39}) can be rewritten as
$R_{n+1}'\rtimes R_n'$, which verifies Condition~(\ref{cond:RR:rtimes}).
\end{itemize}

By Theorem~\ref{thm:criterion:interlacing:01},
there exists a set $T_{n+2}\subseteq R_{n+2}\cap(-\infty,\,x_\Delta)$ such that $|T_{n+2}|=d_{n+2}-1$ and $T_{n+2}\rtimes T_{n+1}$.
From Inequality~(\ref{pf5}), we see that $y_{n+2}>x_\Delta$.
It follows that
\begin{equation}\label{pf70}
R_{n+2}\cap(-\infty,\,x_\Delta)=(R_{n+2}'\cup\{y_{n+2}\})\cap(-\infty,\,x_\Delta)\subseteq R_{n+2}'.
\end{equation}
Thus, we have $T_{n+2}\subseteq R_{n+2}'$.
Since the sets $T_{n+2}$ and $R_{n+2}'$ have the same cardinality $d_{n+2}-1$, 
we infer that $T_{n+2}=R_{n+2}'$.
Now, the result $T_{n+2}\subset (-\infty,\,x_\Delta)$ is one of the desired relations:
\begin{equation}\label{pf71}
R_{n+2}'\subset(-\infty,\,x_\Delta);
\end{equation}
the result $T_{n+2}\rtimes T_{n+1}$ is another one of the desired relations:
\[
R_{n+2}'\rtimes R_{n+1}'.
\]
In view of our goal~(\ref{dsr:thm391}), it suffices to show that $R_{n+2}'\bowtie R_{n}'$,
i.e., $T_{n+2}\bowtie T_n$. By Theorem~\ref{thm:criterion:interlacing:01}, it suffices to verify Condition~(\ref{cond:RR:m+2:01}),
i.e.,
\[
R_{n+2}'=R_{n+2}\cap(-\infty,\,x_\Delta).
\]
In view of Relation~(\ref{pf71}), we deduce that $R_{n+2}'\subseteq R_{n+2}\cap(-\infty,\,x_\Delta)$.
Together with Relation~(\ref{pf70}), we find the above equation, which completes the proof.
\end{proof}

Continuing Example~\ref{eg2:order} and Example~\ref{eg3:order},
we present the approximate values of zeros in the ordered set $R_n=\{\xi_{n,1},\ldots,\xi_{n,d_n}\}$.
\smallskip

\begin{eg}\label{eg2:zero}
This example continues Example~\ref{eg2:order}.
Table~\ref{tab:ii} illustrates that for $n\le8$, we have 
\[
y_n =\max \mR_n \< x_B \= 1/2, \qquad
R_{n+1}\;\rtimes\; R_{n}
\rmand
R_{n+2}\;\bowtie\; R_{n}.
\]
\begin{table}[h]
\begin{center}
\caption{The approximate zeros of $W_n(x)$ ($1\le n\le 8$) in Example~\ref{eg2:order}.}\label{tab:ii}
$\begin{tabu}{|c|c|c|c|c|}
\hline & \xi_{n,\,d_n-3} & \xi_{n,\,d_n-2} & \xi_{n,\,d_n-1} & \xi_{n,\,d_n}=y_n \\
\hline n=1 &  &  &  & 0 \\
\hline n=2 &  &  &  & 0.2500 \\
\hline n=3 &  &  & -1.7807 & 0.2807 \\
\hline n=4 &  &  & -0.2886 & 0.2886 \\
\hline n=5 &  & -4.2912 & 0 & 0.2912 \\
\hline n=6 &  & -1.0218 & 0.1046 & 0.2922 \\
\hline n=7 & -7.5833 & -0.3639 & 0.1547 & 0.2926 \\
\hline n=8 & -1.9561 & -0.1194 & 0.1827 & 0.2927 \\
\hline
\end{tabu}$
\end{center}
\end{table}
\noindent A more careful observation suggests that the second largest root $\xi_{n,\,d_n-1}$ is bounded by the number $x_\Delta=1/4$.  In fact, this is true in general; it motivates Theorem~\ref{thm:bound} below.
\end{eg}
\smallskip

\begin{eg}\label{eg3:zero}
This example continues Example~\ref{eg3:order}.
Table~\ref{tab:iii} illustrates that for $n\le8$,
\[
\xi_{n,\,d_n-1}\<x_\Delta\=-5/4
\qquad\rmand\qquad
y_n\>x_B\=-1.
\]
\begin{table}[h]
\begin{center}
\caption{The approximate zeros of $W_n(x)$ ($1\le n\le 8$) in Example~\ref{eg3:order}.}\label{tab:iii}
$\begin{tabu}{|c|c|c|c|c|}
\hline & \xi_{n,\,d_n-3} & \xi_{n,\,d_n-2} & \xi_{n,\,d_n-1} & \xi_{n,\,d_n} \\
\hline n=1 &  &  &  & 0 \\
\hline n=2 &  &  &  & -0.5000 \\
\hline n=3 &  &  & -2.6180 & -0.3819 \\
\hline n=4 &  &  & -1.5773 & -0.4226 \\
\hline n=5 &  & -5.1819 & -1.4064 & -0.4116 \\
\hline n=6 &  & -2.2405 & -1.3444 & -0.4149 \\
\hline n=7 & -8.5525 & -1.7194 & -1.3140 & -0.4139 \\
\hline n=8 & -3.1548 & -1.5342 & -1.2966 & -0.4142 \\
\hline
\end{tabu}$
\end{center}
\end{table}
\noindent A more careful observation suggests that the largest root $y_n$ converges to the point~$x_g$ in an oscillating manner, which is approximately $-0.4142$. In fact, this convergence is true in general; see Theorem~\ref{thm:bound} and Theorem~\ref{thm:lim:xd}.
\end{eg}

\bigskip
%%%%%%%%%%%%%%%%%%%%%%%%%%%%%%%%%%%%%%%%%%%%%
%%%%%%%%%%%%%%%%%%%%%%%%%%%%%%%%%%%%%%%%%%%%%
\section{\large Bound on the Zero-Set $R_n$ \label{sec:bzs}}

As consequence of the real-rootedness of the $(0,1)$-sequence polynomials $\{W_n(x)\}_{n\geq0}$, we improve the bound of the zero-set $R_n$ of $W_n(x)$.

\begin{thm}\label{thm:bound}
Let $\{W_n(x)\}_{n\ge0}$ be a $(0,1)$-sequence of polynomials.
Let us denote the ordered zero-set of $W_n(x)$ by $R_n$, and we let $R_n'=R_n\setminus\{y_n\}$, where $y_n=\max R_n$ is the largest real root of the polynomial~$W_n(x)$.
\begin{itemize}
\medskip\item[(i)]
If $c\le-(a^2+2ab)/4$, then $R_n\subset (-\infty,x_\Delta)$ for all $n\ge 1$.
\medskip\item[(ii)]
If $-(a^2+2ab)/4<c<0$, then we have
\begin{itemize}
\smallskip\item
$R_n\subset  (-\infty,x_\Delta)$, for $n<n_0$;
\smallskip\item
$R_n'\subset  (-\infty,x_\Delta)$ and $y_n=x_\Delta$, for $n=n_0$;
\smallskip\item
$R_n'\subset  (-\infty,x_\Delta)$ and $y_n\in(x_\Delta,\,x_g)$, for $n>n_0$;
\end{itemize}
\smallskip
\medskip\item[(iii)]
If $c>0$, then we have $R_n'\subset  (-\infty,x_\Delta)$, and
\begin{equation}\label{ineq:osc}
x_B<y_2<y_4<y_6<
\cdots<y_{2n}<\cdots<x_g<\cdots<y_{2n-1}<\cdots<y_5<y_3<y_1=0.
\end{equation}
\end{itemize}
\end{thm}
\begin{proof}
We treat the three cases individually.

\noindent{\bf (i)} Let $c\le-(a^2+2ab)/4$.
Since $a,b>0$, it is routine to check that
$$-c/(a+b)\<-(a^{2}+4c)/(4b)\<-c/b,$$
which verifies Condition~(\ref{cond:k=0:01}) for $\beta=x_\Delta$.
By Lemma~\ref{lem:order:01}, we have
\begin{equation}\label{pf3102}
W_n(x_\Delta)>0\all{n\ge1}.
\end{equation}
Now, by Lemma~\ref{lem:k=0:01}, we deduce that $R_{n}\subset (-\infty,\,x_\Delta)$ for all $n\ge1$.

\medskip\noindent{\bf (ii)} Let $-(a^2+2ab)/4<c<0$.

\smallskip

{\bf Case $n<n_0$.}
Recall from Formula~(\ref{R1R2:01}) that $R_1=\{0\}$ and $R_2=\{-c/(a+b)\}$.  If $n_0\le 1$, then nothing needs to be shown in this case. Next suppose that $n_0>1$, i.e., $a^2+4c<0$.  Together with $b>0$, this implies that $0<-(a^2+4c)/(4b)=x_\Delta$, i.e., $R_1\subset (-\infty,\,x_\Delta)$.  If $n_0\le 2$, then nothing else needs to be shown.  And then suppose that $n_0>2$, i.e., $a^2+ab+4c<0$.  Together with $a,b>0$, it is routine to check that
\begin{equation}\label{pf73}
-c/(a+b)<-(a^2+4c)/(4b),
\end{equation}
i.e., $R_2\subset (-\infty,\,x_\Delta)$.
If $n_0\le 3$, nothing else needs to be shown. So we may suppose that $n_0>3$.

Let $N=\lceil{n_0\rceil}-3$. Since $n_0>3$, the integer~$N$ is positive.
Take $\beta=x_\Delta$. From $x_\Delta<-c/b$, together with Inequality~(\ref{pf73}),
we see that Condition~(\ref{cond:k=0:01}) holds true.
By Lemma~\ref{lem:order:01}, we have $W_n(x_\Delta)>0$ for all $n\in[N]$.
By Lemma~\ref{lem:k=0:01}, we have $R_{n}\subset(-\infty,\,x_\Delta)$ for all $n\in[N+2]=[\lceil n_0\rceil -1]$, i.e.,
for all $n<n_0$.

\smallskip
{\bf Case $n=n_0$}. It follows that the number $n_0$ is an integer.
By Lemma~\ref{lem:order:01}, we have $W_{n_0}(x_\Delta)=0$.
It suffices to show that the polynomial $W_{n_0}(x)$ has no roots larger than the number~$x_\Delta$.
If $n_0=1$, then the polynomial $W_{n_0}(x)=W_1(x)=x$ has only one root. So we are done.
Suppose that $n_0\ge 2$.
By the interlacing property $R_{n_0}\rtimes R_{n_0-1}$ obtained in Theorem~\ref{thm:RR:01},
we see that the second largest root of the polynomial $W_{n_0}(x)$ is less than the largest
root of the polynomial $W_{n_0-1}(x)$, which is less than the number~$x_\Delta$, in view of the case $n<n_0$.
This completes the proof for the case $n=n_0$.

\smallskip
{\bf Case $n>n_0$.}
First, we show that $y_n<x_g$, i.e., $R_n\subset(-\infty,\,x_g)$.
We do this by applying Lemma~\ref{lem:fg:01} for $\beta=x_g$.
Recall from Definition~(\ref{def:xB:xd:xg}) that $x_g=(a+b-\sqrt{(a+b)^2+4c})/2$.
Since $a,b>0$ and $-(a^2+2ab)/4<c<0$, it is routine to check that
\begin{equation}\label{pf76}
-c/(a+b)\<(a+b-\sqrt{(a+b)^2+4c})/2.
\end{equation}
By Lemma~\ref{lem:order:01} (ii), we have
\begin{equation}\label{pf75}
\max(0,\,x_\Delta)<x_g<x_B.
\end{equation}
The particular inequality $x_g<x_B$, together with Inequality~(\ref{pf76}), verifies Condition~(\ref{cond:k=0:01}).
On the other hand, since $x_g>0$, Formula~(\ref{W:xg}) implies that
\begin{equation}\label{sgn:xg:ii}
W_n(x_g)\>0, \all{n\ge1}.
\end{equation}
From Lemma~\ref{lem:k=0:01}, we deduce that $R_n\subset (-\infty,\,x_g)$ for all $n\ge1$. 

By Lemma~\ref{lem:order:01}, we have $W_n(x_\Delta)<0$.
In view of Inequality~(\ref{sgn:xg:ii}),
the polynomial $W_n(x)$ has different signs at the ends of the interval $(x_\Delta,\,x_g)$.
Therefore,
the polynomial $W_n(x)$ has an odd number, say $p_n$, of roots in the interval~$(x_\Delta,\,x_g)$.
In particular, we have
\begin{equation}\label{pf:pn:01}
p_n\ge1\all{n>n_0}.
\end{equation}
It suffices to show that $p_n=1$, for all $n>n_0$. We proceed the proof by induction on $n$.
Note that the largest root of the polynomial $W_{\fl{n_0}}(x)$ is less than or equal to the number~$x_\Delta$.
By the interlacing property $R_{\fl{n_0}+1}\rtimes R_{\fl{n_0}}$,
the polynomial $W_{\fl{n_0}+1}(x)$ has at most one root larger than the number $x_\Delta$,
i.e., $p_{\fl{n_0}+1}\le 1$. In view of Inequality~(\ref{pf:pn:01}), we deduce that $p_{\fl{n_0}+1}=1$.
Thus, we can suppose that $p_n=1$ for some $n>n_0$.
If $n\le 2$, then the degree $d_n\le 1$. It follows immediately that $p_n=1$.
Suppose that $n\ge 3$.
By the interlacing property $R_{n+1}\rtimes R_n$,
the third largest root of the polynomial $W_{n+1}(x)$ is less than the second largest root of the polynomial~$W_n(x)$,
which is at most $x_\Delta$ since $p_n=1$.
Therefore, the polynomial $W_{n+1}(x)$ has at most two roots larger than the number~$x_\Delta$, i.e., $p_n\le 2$.
Since the integer~$p_n$ is odd,
in view of Inequality~(\ref{pf:pn:01}), we infer that $p_n=1$. This completes and the induction and hence the proof of (ii).

\medskip\noindent{\bf(iii)} Let $c>0$.
The bound for the set~$R_n'$ has been confirmed in Theorem~\ref{thm:RR:01}.
It suffices to show Inequality~(\ref{ineq:osc}). By Theorem~\ref{thm:RR:01}, we have $y_n>x_B$ for all $n\ge1$.
It suffices to show that
\begin{align}
&y_{2n}\<y_{2n+2}\<x_g\qquad\text{and}\quad\label{dsr:y:even}\\[3pt]
&x_g\<y_{2n+1}\<y_{2n-1}\label{dsr:y:odd}
\end{align}
for all $n\ge 0$, where $y_0=x_B$ and $y_{-1}=+\infty$. We proceed by induction on the integer~$n$.
When $n=0$, the desired inequalities~(\ref{dsr:y:even}) and~(\ref{dsr:y:odd}) become
$y_2<x_g<y_1$, i.e.,
\[
-c/(a+b)<(a+b-\sqrt{(a+b)^2+4c})/2<0.
\]
Since $a,b,c>0$, it is routine to check the truth of the above inequalities.
Now, based on the induction hypothesis that
\begin{equation}\label{pf:hypo}
y_{2n}\<x_g\<y_{2n-1},
\end{equation}
we are going to show the inequalities~(\ref{dsr:y:even}) and~(\ref{dsr:y:odd}).

Since the number~$y_{2n}$ is largest real root of the polynomial $W_{2n}(x)$,
and $y_{2n-1}>y_{2n}$ by the hypothesis~(\ref{pf:hypo}),
we infer that the value $W_{2n}(y_{2n-1})$ has the same sign as the limit $W_{2n}(+\infty)$,
which is positive by Lemma~\ref{lem:deg:01}.
Therefore, we find $W_{2n}(y_{2n-1})>0$.
Replacing $n$ by $2n-1$ in Recursion~(\ref{rec:01}), and taking $x=y_{2n-1}$, we obtain that
\begin{equation}\label{pf51}
W_{2n+1}(y_{2n-1})=aW_{2n}(y_{2n-1})>0.
\end{equation}
On the other hand,
by Lemma~\ref{lem:order:01}, we have $x_g<0$. Thus from Equation~(\ref{W:xg}), we infer that
\begin{equation}\label{sgn:xg:iii}
W_n(x_g)(-1)^n\>0, \all{n\ge1}.
\end{equation}
In particular, we have $W_{2n+1}(x_g)<0$.
Together with~(\ref{pf51}), we see that the polynomial $W_{2n+1}(x)$
attains different signs at the ends of the interval $(x_g,\,y_{2n-1})$.
By the intermediate value theorem, the polynomial $W_{2n+1}(x)$ has a root in the interval $(x_g,\,y_{2n-1})$.
By Theorem~\ref{thm:RR:01}, only the largest root~$y_{2n+1}$ of the polynomial $W_{2n+1}(x)$ is larger than the number~$x_B$.
Since $x_B<x_g$, we conclude that $y_{2n+1}\in(x_g,\,x_{2n-1})$. This proves Inequality~(\ref{dsr:y:odd}).

Denote by $z_{2n+1}$ the second largest root of the polynomial $W_{2n+1}(x)$.
From the interlacing property $R_{2n+1}\rtimes R_{2n}$, we infer that
\[
W_{2n+1}(x)W_{2n+1}(+\infty)<0\all{x\in(z_{2n+1},\,y_{2n+1})}.
\]
By Lemma~\ref{lem:deg:01}, we see that the limit $W_{2n+1}(+\infty)=+\infty$. It follows that
\begin{equation}\label{pf52}
W_{2n+1}(x)<0\all{x\in(z_{2n+1},\,y_{2n+1})}.
\end{equation}
Now, from Inequality~(\ref{dsr:y:odd}) and the hypothesis~(\ref{pf:hypo}),
we see that $y_{2n}<x_g<y_{2n+1}$. From Theorem~\ref{thm:RR:01}, we see that $z_{2n+1}<x_B<y_{2n}$.
By Inequality~(\ref{pf52}), we infer that $W_{2n+1}(y_{2n})<0$.

Replacing $n$ by $2n+2$ in Recursion~(\ref{rec:01}), and taking $x=y_{2n}$, we obtain that
\begin{equation}\label{pf53}
W_{2n+2}(y_{2n})=aW_{2n+1}(y_{2n})<0.
\end{equation}
By Inequality~(\ref{sgn:xg:iii}), we have $W_{2n+2}(x_g)>0$.
By the intermediate value theorem,
the polynomial $W_{2n+2}(x)$ has a root in the interval $(y_{2n},\,x_g)$.
Since only its largest root is larger than the number~$x_B$, and since~$y_{2n}>x_B$,
we conclude that $y_{2n+2}\in(y_{2n},\,x_g)$.
This proves Inequality~(\ref{dsr:y:even}), which completed the induction.
\end{proof}

In summary, we see that ``almost all'' zeros lie in the open interval $(-\infty,x_\Delta)$. Precisely speaking, when $c\le-(a^2+2ab)/4$, all roots lie in $(-\infty,x_\Delta)$; when $c>-(a^2+2ab)/4$, only the largest root of the polynomial~$W_n(x)$ is possibly but ``eventually'' larger than $x_\Delta$, with maximum value $\max(x_g,\,0)$.

Before ending this section, we mention that the recurrence system defined by Recursion~(\ref{Mrec:W:01}) can be solved always by transforming the polynomials $W_n(x)$ into Chebyshev polynomials. More precisely, by induction and by the fact that Chebyshev polynomials of the second kind satisfy the recursion $U_n(t)=2tU_{n-1}(t)-U_{n-2}(t)$ with initial conditions $U_0(t)=1$ and $U_1(t)=t$, we obtain that
$$W_n(x)=\sqrt{-bx-c}^{\,n}\left(\frac{x}{\sqrt{-bx-c}}U_{n-1}\left(\frac{a}{2\sqrt{-bx-c}}\right)-U_{n-2}\left(\frac{a}{2\sqrt{-bx-c}}\right)\right).$$
By this, it is now clear that all roots of $W_n(x)$ are real and bounded, as described in Theorems \ref{thm:RR:01} and \ref{thm:bound}.

\bigskip
%%%%%%%%%%%%%%%%%%%%%%%%%%%%%%%%%%%%%%%%%%%%%
%%%%%%%%%%%%%%%%%%%%%%%%%%%%%%%%%%%%%%%%%%%%%
\section{\large Limit Points of the Zero-Set $R_n$ \label{sec:limpts}}

In this section, we show that one of the intervals~$(-\infty,x_\Delta)$, $(-\infty,\,x_g)$,
and $(-\infty,\,y_2)$ is the best bound of all zeros, depending on the range of the constant term~$c$ of the linear polynomial coefficient~$B(x)=bx+c$. More precisely, we will demonstrate the aforementioned three limit points of the zero-set $\cup_{n\ge1}R_n$ over the course of several subsections. We say that a proposition \emph{holds for large $n$}, if there exists a number~$N$ such that the proposition holds whenever $n>N$.
\medskip

\subsection{The number $x_g$ can be a limit point}
The following lemma will help determine all limit points of the zero-set
$\cup_{n\ge1}R_n$, which are larger than the number $x_\Delta$.

\begin{lem}\label{thm:limsgn:d>0:01}
Let $\{W_n(x)\}_{n\ge0}$ be a $(0,1)$-sequence of polynomials. Let $x_0\ne x_g$ and $\Delta(x_0)>0$. Then
$(x_0-x_g)W_n(x_0)>0$, for large $n$.
\end{lem}
\begin{proof}
See Appendix \ref{proofthm:limsgn:d>0:01}.
\end{proof}
\smallskip

Using Lemma~\ref{thm:limsgn:d>0:01},
we can confirm that the roots outside the interval~$(-\infty,x_\Delta)$ converges to
the number~$x_g$ when $n\to\infty$ as follows.
\smallskip

\begin{thm}\label{thm:lim:xg}
Let $\{W_n(x)\}_{n\ge0}$ be a $(0,1)$-sequence of polynomials. Let us denote the ordered zero-set of $W_n(x)$ by $R_n$, and we let $y_n=\max R_n$ be the largest real root of the polynomial~$W_n(x)$.
\begin{itemize}
\smallskip\item[(i)]
If $-(a^2+2ab)/4<c<0$, then we have $y_n\nearrow x_g$.
\smallskip\item[(ii)]
If $c>0$, then we have $y_{2n}\nearrow x_g$ and $y_{2n+1}\searrow x_g$.
\end{itemize}
\end{thm}

\begin{proof}
We treat the two cases individually.

\medskip\noindent{\bf(i)} Suppose that $-(a^2+2ab)/4<c<0$.
Since $R_{n+1}\rtimes R_n$, the sequence $y_n$ increases.
In virtue of Theorem~\ref{thm:bound},
we have $y_n<x_g$ for all $n\ge1$.
Therefore, the sequence $y_n$ converges to a finite number $y^*$ as $n\to\infty$.
If $y^*<x_g$, then there exists $x_0\in(x_\Delta,\,x_g)$ such that
the values~$W_n(x_0)$ and $W_n(x_g)$ have the same sign for large~$n$,
i.e., $W_n(x_0)>0$ for large~$n$; see Inequality~(\ref{sgn:xg:ii}).
This contradicts Theorem~\ref{thm:limsgn:d>0:01}.
Hence, we have that $y_n\nearrow x_g$.

\medskip\noindent{\bf(ii)} Suppose that $c>0$.
From Theorem~\ref{thm:bound},
we see that the sequence $y_{2n}$ converges to a finite number $y^*$.
Then we have $x_B<y^*\le x_g$.
Suppose to the contrary that $y^*<x_g$, so there exists $x_0\in(\ell_e,\,x_g)$ such that
the numbers $W_{2n}(x_0)$ and $W_{2n}(x_g)$ have the same sign for large $n$,
i.e., $W_{2n}(x_0)>0$ for large $n$; see Inequality~(\ref{sgn:xg:iii}).
This contradicts Theorem~\ref{thm:limsgn:d>0:01}.
Along the same line, we can show the convergence $y_{2n+1}\searrow x_g$, which completes the proof.
\end{proof}
\smallskip

An illustration for the convergences above can be found in Tables~\ref{tab:ii} and \ref{tab:iii}.
\medskip

\subsection{The number $x_\Delta$ is a limit point}

In an analog with Lemma~\ref{thm:limsgn:d>0:01},
we give a characterization of the sign of the value~$W_n(x_0)$
for the case $\Delta(x_0)<0$.
This time the criterion for the sign is for all positive integers $n$.
We define $l_{x_0}$ to be the straight line $\sqrt{-\Delta(x_0)}\,x+(2x_0-a)y=0$, and the radian $\theta(x_0)$ to be $\arctan{\sqrt{-\Delta(x_0)}\over a}$.

\eject

\begin{lem}\label{thm:sgn:d<0:01}
Let $\{W_n(x)\}_{n\ge0}$ be a $(0,1)$-sequence of polynomials, and let $\Delta(x_0)<0$.
\begin{itemize}
\item If the radian $n\theta(x_0)$ lies to the left of the line $l_{x_0}$, then $W_n(x_0)<0$;
\item If the radian $n\theta(x_0)$ lies on the line $l_{x_0}$, then $W_n(x_0)=0$;
\item If the radian $n\theta(x_0)$ lies to the right of the line $l_{x_0}$, then $W_n(x_0)>0$.
\end{itemize}
\end{lem}
\begin{proof}
See Appendix \ref{proofthm:sgn:d<0:01}.
\end{proof}

Let us get some illustration of this characterization from the example below.

\begin{eg}\label{eg2:sgn:d<0}
This example continues Example~\ref{eg2:order}.
Take $x_0=-1$, we have $\Delta(x_0)=-5<0$ and $\theta(x_0)=\arctan(\sqrt{5})$.
The line $l_{x_0}$ becomes $\sqrt{5}x-3y=0$.
Thus a radian $\phi$ lies to the left of the line $l_{x_0}$ if and only if
\begin{equation}\label{eg:range}
\phi\in\bigl(\arctan(\sqrt{5}/3)+2\ell\pi,\,\arctan(\sqrt{5}/3)+(2\ell+1)\pi\bigr)
\quad\text{for some integer $\ell$}.
\end{equation}
By approximating $\arctan(\sqrt{5}/3)\approx0.6405$, we have that
$\theta(x_0)\approx 1.1502$, $2\theta(x_0)\approx 2.3005$ and
$3\theta(x_0)\approx 3.4507$.
By Relation~(\ref{eg:range}),
we deduce that $\theta_{x_0}$, $2\theta_{x_0}$ and~$3\theta_{x_0}$
lie to the left of the line~$l_{x_0}$.
In the same way we can deduce that the radians $4\theta(x_0)\approx4.6010$,
$5\theta(x_0)\approx5.7513$ and $6\theta(x_0)\approx6.9015$
lie to the right of the line~$l_{x_0}$.
The truth is, as one may compute directly, that
\begin{xalignat*}{3}
W_1(-1)&\=-1,
& W_2(-1)&\=-5/2,
& W_3(-1)&\=-1,\\[3pt]
W_4(-1)&\=11/4,
& W_5(-1)&\=17/4,
& W_6(-1)&\=1/8.
\end{xalignat*}
The above data verifies the fact that $W_n(x_0)<0$ for $n\in\{1,2,3\}$,
and that $W_n(x_0)>0$ for $n\in\{4,5,6\}$, coinciding with the characterization.
\end{eg}

Now we are ready to justify that the number $x_\Delta$ is a limit point.

\begin{thm}\label{thm:lim:xd}
Let $\{W_n(x)\}_{n\ge0}$ be a $(0,1)$-sequence of polynomials, and let us denote the ordered zero-set of $W_n(x)$ by $R_n=\{\xi_{n,1},\ldots,\xi_{n,d_n}\}$. Then
\begin{equation}\label{lim:xd}
\lim_{n\to\infty}\xi_{n,\,d_n-i}\=x_\Delta
\end{equation}
for all $i\ge0$ if $c\le -(a^2+2ab)/4$; and for all $i\ge1$ otherwise.
\end{thm}
\begin{proof}
Let $c\le -(a^2+2ab)/4$.
We will show the limit~(\ref{lim:xd}) for all $i\ge0$.
As will be seen, the other case can be done in the same vein.

From the interlacing property obtained in Theorem~\ref{thm:RR:01},
we see that the sequence $\{\xi_{n,\,d_n-i}\}_{n\ge1}$ increases and all its members are less than
the number~$x_\Delta$, which implies that it converges to a number which is at most~$x_\Delta$.
Suppose, by way of contradiction,
that the limit point of the sequence~$\{\xi_{n,\,d_n-i}\}_{n\ge1}$ is not the point~$x_\Delta$.

When $i=0$, there exists a point $x_0<x_\Delta$ such that
the numbers $W_n(x_0)$ and $W_n(x_\Delta)$ have the same sign,
i.e.,
we have $W_n(x_0)>0$ for large~$n$.
Therefore, by Lemma~\ref{thm:sgn:d<0:01}, the radian~$n\theta(x_0)$ resides
in certain one side of the line~$l_{x_0}$ forever for large~$n$.
This is impossible because $\theta(x_0)<\pi/2$. Hence we deduce that $\lim_{n\to\infty}\xi_{n,d_n}\=x_\Delta$.

Now for $i=1$, the sequence~$\{\xi_{n,\,d_n-1}\}_{n\ge1}$ converges to some point less than $x_\Delta$.
Thus, there exists a number $x_1<x_\Delta$ such that
the numbers $W_n(x_1)$ and $W_n(x_\Delta)$ have distinct signs,
i.e.,
we have $W_n(x_1)<0$ for large $n$.
Here again, the radian~$\theta(x_1)$ resides
in certain one side of the line~$l_{x_1}$ for large~$n$, a contradiction.
This confirms the truth of the limit~(\ref{lim:xd}) for $i=1$.
Continuing in this way,
we can deduce that for a general $i\ge2$,
there exists a number $x_i<x_\Delta$, such that
\[
W_n(x_i)(-1)^i\>0\qquad\text{for large $n$},
\]
which contradicts Lemma~\ref{thm:sgn:d<0:01}.
Hence, we conclude that the limit~(\ref{lim:xd}) holds true for all $i\ge0$.

Now we consider the other possibility that $c>-(a^2+2ab)/4$.
In fact, the above contradiction idea still works.
This is because that, whatever sign does the value~$W_n(x_\Delta)$ have,
it is a fixed sign.
However, the sign of the value~$W_n(x_0)$
for any point $x_0<x_\Delta$ can not be invariant for large $n$.
This completes the proof.
\end{proof}
\medskip

\subsection{Negative infinity is a limit point}
Now we are ready to study the negative infinity as a limit point.
\smallskip

\begin{thm}\label{thm:lim:-infty}
Let $\{W_n(x)\}_{n\ge0}$ be a $(0,1)$-sequence of polynomials, and let us denote the ordered zero-set of $W_n(x)$ by $R_n=\{\xi_{n,1},\ldots,\xi_{n,d_n}\}$. Then
$\lim_{n\to\infty}\xi_{n,i}\=-\infty$, for all positive integers $i$.
\end{thm}
\begin{proof}
From the interlacing property $R_{n+2}\bowtie R_n$ obtained in Theorem~\ref{thm:RR:01},
we see that the sequences~$\{\xi_{2n,\,i}\}_{n\ge1}$ decreases,
and so does the sequence~$\{\xi_{2n-1,\,i}\}_{n\ge1}$.
Therefore, these two sequences converge respectively.
We shall show that both of these sequences converge to the negative infinity.

Suppose, by way of contradiction, that the sequence $\{\xi_{2n,\,1}\}_{n\ge1}$
converges to some real number~$x^*$. Then for any number $x_0<x^*$,
the number $W_n(x_0)$ has the sign of $W_n(-\infty)$.
It follows that the sign of the number~$W_n(x_0)$ would not change for large~$n$,
which contradicts Theorem~\ref{thm:sgn:d<0:01}. This proves that
$\lim_{n\to\infty}\xi_{2n,i}\=-\infty$
for $i=1$. Its truth for general $i$, in fact, along the same lines,
if it does not hold for some $i\ge2$,
then we can deduce the existence of a number $x_i$ such that $x_i<x_\Delta$
and that the sign of the number $W_n(x_i)$ keeps invariant for large $n$, which leads to a contradiction.

Along the same lines, we can prove that $\lim_{n\to\infty}\xi_{2n-1,\,i}\=-\infty$, for all $i\ge1$.

Now, for any fixed $i\ge1$,
the subsequences $\{\xi_{2n,\,i}\}_{n\ge1}$ and $\{\xi_{2n-1,\,i}\}_{n\ge1}$ converge
to the same point $-\infty$. Hence, the joint sequence $\{\xi_{n,i}\}_{n\ge 1}$
converges to the negative infinity as well, which completes the proof.
\end{proof}
\smallskip
For an illustration for the convergences in Theorem~\ref{thm:lim:xd}
and Theorem~\ref{thm:lim:-infty},
the reader can refer to Tables~\ref{tab:ii} and \ref{tab:iii}.

\bigskip
%%%%%%%%%%%%%%%%%%%%%%%%%%%%%%%%%%%%%%%%%%%%%
%%%%%%%%%%%%%%%%%%%%%%%%%%%%%%%%%%%%%%%%%%%%%
%--------------------------------------------------------------------------------------------

\bigskip
%%%%%%%%%%%%%%%%%%%%%%%%%%%%%%%%%%%%%%%%%%%%%
%%%%%%%%%%%%%%%%%%%%%%%%%%%%%%%%%%%%%%%%%%%%%
\section{\large Appendix: Technical proofs}

\subsection{Proof of Lemma \ref{lem:fg:01}}\label{prooflem:fg:01}
Let $x_0=y_0=-\infty$ and $y_{q+1}=\beta$.
The interlacing property $X'\rtimes Y'$ in the premises implies that $p\ge1$ and $q\in\{p-1,\,p\}$.
Since $X'\subset(-\infty,\,\beta)$, we infer that $x_p<\beta$. Therefore, we have that
\[
\cdots\<y_{q-2}\<x_{p-2}\<y_{q-1}\<x_{p-1}\<y_q\<x_p\<\beta.
\]
We shall show Inequality~(\ref{ineq:lem:f}) and Inequality~(\ref{ineq:lem:g}) respectively.

Let $i\in[p]$.  From the definition $X'=X\cap(-\infty,\,\beta)$ and the interlacing property $X'\rtimes Y'$ in the premises, we see that the number~$x_{p+1-i}$ is the unique root of the polynomial~$f(x)$ in the interval $(y_{q+1-i},\,y_{q+2-i})$. Suppose that $f(\beta)\neq0$.
By the intermediate value theorem, we infer that
$f(y_{q+1-i})f(y_{q+2-i})<0$, that is,
\begin{align*}
f(y_{q})f(\beta)&\<0\qquad(i=1),\\
f(y_{q-1})f(y_{q})&\<0\qquad(i=2),\\
&\quad\vdots\\
f(y_{q-p+1})f(y_{q-p+2})&\<0\qquad(i=p).
\end{align*}
Multiplying the first $i$ inequalities in the above list results in that
\[
f(y_{q+1-i})f(\beta)(-1)^{i-1}<0.
\]
Replacing $i$ by $q+1-j$ in it yields Inequality~(\ref{ineq:lem:f}) for $j\in[q+1-p,\,q]$.
When $j=q+1$, since $y_{q+1}=\beta$ stands as a premise, Inequality~(\ref{ineq:lem:f}) holds true trivially.

From the definition $Y'=Y\cap(-\infty,\,\beta)$,
we deduce that the polynomial $g(x)$ has no roots in the interval $(y_q,\beta)$.
Suppose that $g(\beta)\ne0$.
By the intermediate value theorem,
we infer that $g(x)g(\beta)>0$ for all $x\in(y_q,\beta)$. In particular,
we have
\begin{equation}\label{pf56}
g(x_p)g(\beta)\>0,
\end{equation}
which is Inequality~(\ref{ineq:lem:g}) for $j=p$.
Below we can suppose that $p\ge2$, and thus, $q\ge1$.

Let $j\in[q]$. Similar to the previous proof,
we have $g(x_{p-j})g(x_{p+1-j})<0$, that is,
\begin{align*}
g(x_{p-1})g(x_p)&\<0\qquad(j=1),\\
g(x_{p-2})g(x_{p-1})&\<0\qquad(j=2),\\
&\quad\vdots\\
g(x_{p-q})g(x_{p-q+1})&\<0\qquad(j=q).
\end{align*}
Multiplying the first $j$ inequalities in the above list, we find that
\[
g(x_{p-j})g(x_{p})(-1)^{j-1}\<0.
\]
Multiplying it by Inequality~(\ref{pf56}) results in that $g(x_{p-j})g(\beta)(-1)^{j-1}\<0$. Replacing $j$ by $p-i$ in it yields that
\[
g(x_{i})g(\beta)(-1)^{p-i}\>0.
\]
Together with Inequality~(\ref{pf56}), we obtain Inequality~(\ref{ineq:lem:g}). This completes the proof.\qed
\medskip

\subsection{Proof of Lemma \ref{lem:order:01}}\label{prooflem:order:01}

From Equation~(\ref{W:xd}), we have that
\begin{equation}\label{pf1}
\mbox{the numbers }W_n(x_\Delta)\mbox{ and }a+n(2x_\Delta-a)\mbox{ have the same sign}.
\end{equation}
\noindent{\bf(i)} If $c\le -(a^2+2ab)/4$, then we have
\begin{equation}\label{pf2}
x_\Delta=-{a^2+4c\over 4b}\ge-{a^2-(a^2+2ab)\over 4b}={a\over 2},
\end{equation}
that is, $2x_\Delta-a\ge 0$. It follows that $a+n(2x_\Delta-a\ge 0)>0$ for all $n\ge 1$.
By Relation~(\ref{pf1}), we obtain that $W_n(x_\Delta)>0$.

\noindent{\bf(ii)} Below we suppose that $c>-(a^2+2ab)/4$.
From the deduction~(\ref{pf2}), we see that $2x_\Delta-a<0$.
If $n<n_0$, then we have
$a+n(2x_\Delta-a)>a+{2ab\over a^2+2ab+4c}\cdot\bigl(2(-{a^2+4c\over 4b})-a\bigr)=0$, which, by (\ref{pf1}), implies that $W_n(x_\Delta)>0$.
Similarly, if $n=n_0$ then we have that $a+n(2x_\Delta-a)=0$, and thus $W_n(x_\Delta)=0$ by the relation~(\ref{pf1}); and
if $n>n_0$ then we have that $a+n(2x_\Delta-a)<0$, and thus $W_n(x_\Delta)<0$ by the relation~(\ref{pf1}).

When $c>0$, we have that $n_0={2ab\over a^2+2ab+4c}<1$. Therefore, the case $n>n_0$ happens for all $n\geq 1$, that is, $W_n(x_\Delta)<0$.
Thus, by  Definition \ref{def:xg} we obtain that $x_g<0$. Moreover, we have
\[
x_g-x_B={(a+b)-\sqrt{(a+b)^2+4c}\over 2}+{c\over b}
={ab+b^2+2c-b\sqrt{a^2+2ab+b^2+4c}\over 2b}.
\]
Thus, to show that $x_g>x_B$, it suffices to show that $(ab+b^2+2c)^2>b^2(a^2+2ab+b^2+4c)$.
By direct calculation, this inequality is equivalent to $4c(ab+c)>0$, which is true since $a,b,c>0$.

When $c<0$, we have $x_g>0$ from Definition~(\ref{def:xg}) straightforwardly.
Suppose to the contrary that $x_g\ge x_B$. It follows that $2c+b(a+b)\ge b\sqrt{(a+b)^2+4c}>0$.
Solving $(2c+b(a+b))^2\ge (b\sqrt{(a+b)^2+4c})^2$ with $c<0$, we see that $c\le -ab$. On the one hand, by solving $2c+b(a+b)>0$, we get $-b(a+b)/2<c\le -ab$, which implies that $a<b$. On the other hand, we have $-(a^2+2ab)/4<c\le -ab$, which implies that $a>2b$. Hence, $2b<a<b$, a contradiction. This proves $x_g<x_B$ when $-(a^2+2ab)/4<c<0$.\qed
\medskip

\subsection{Proof of Lemma \ref{thm:limsgn:d>0:01}}\label{proofthm:limsgn:d>0:01}

By Lemma~\ref{lem:00}, the value $W_n(x_0)$ can be recast as the following form
\[
W_n(x_0)
\={\bigl(A(x_0)+\sqrt{\Delta(x_0)}\,\bigr)^n\over 2^{n}\sqrt{\Delta(x_0)}}
\left[g^+(x_0)-g^-(x_0)\bgg{{A(x_0)-\sqrt{\Delta(x_0)}\over A(x_0)+\sqrt{\Delta(x_0)}}\,}^n\right].
\]
Since $A(x_0)=a>0$ and $\sqrt{\Delta(x_0)}>0$, we deduce that
\[
\left|{A(x_0)-\sqrt{\Delta(x_0)}\over A(x_0)+\sqrt{\Delta(x_0)}}\right|\<1.
\]
Thus we obtain that
\begin{equation}\label{sgn:W:g+}
W_n(x_0)g^+(x_0)>0\qquad\text{for large $n$}.
\end{equation}
Note that the function $2g^+(x)\=2x-a+\sqrt{4(bx+c)+a^2}$
is increasing. Since $g^+(x_g)=0$, we infer that $(x_0-x_g)g^+(x_0)>0$.
In view of~(\ref{sgn:W:g+}), we conclude that $(x_0-x_g)W_n(x_0)>0$
for large $n$, which completes the proof.\qed
\medskip

\subsection{Proof of Lemma \ref{thm:sgn:d<0:01}}\label{proofthm:sgn:d<0:01}

By Lemma~\ref{lem:00}, the sign of the value~$W_n(x_0)$ is equal to the sign of the value $F=\cos\theta+\ell\sin\theta$,
where $\theta=n\theta(x_0)$, and
$\ell=(2x_0-a)/\sqrt{-\Delta(x_0)}$.

If $x_0=a/2$, then the line $l_{x_0}$ becomes the imagine axis $x=0$.
In this case, the sign of the value $W_n(x_0)$
is determined by the sign of the value~$\cos\theta$.
In other words, we have $W_n(x_0)>0$ if and only if
the radian~$n\theta_0$ lies in the right open half-plane, and
$W_n(x_0)<0$ if and only if
the radian~$n\theta_0$ lies in the left open half-plane.

Below we can suppose that $x_0\ne a/2$. It follows that $\ell\ne 0$.
\begin{itemize}
\item Assume that $\ell>0$.
It is elementary to find the following equivalence relation
\[
F>0
\eqrl
\begin{cases}
\tan\theta>-1/\ell,&\text{if $\cos\theta>0$};\\[5pt]
\sin\theta>0,&\text{if $\cos\theta=0$};\\[5pt]
\tan\theta<-1/\ell,&\text{if $\cos\theta<0$}.
\end{cases}
\]
In this case, we have $F>0$ if and only if the radian~$\theta$ lies to
the right of the line $y=-x/\ell$, that is, of the line $l_{x_0}$.
By symmetry, we have $F<0$ if and only if the radian~$\theta$ lies to
the left of the line~$l_{x_0}$. It follows immediately that $F=0$
if and only if the radian~$\theta$ lies on the line~$l_{x_0}$.

\item Now suppose that $\ell<0$.
Then we have the following equivalence relation in the same vein:
\[
F>0
\eqrl
\begin{cases}
\tan\theta<-1/\ell,&\text{if $\cos\theta>0$};\\[5pt]
\sin\theta<0,&\text{if $\cos\theta=0$};\\[5pt]
\tan\theta>-1/\ell,&\text{if $\cos\theta<0$}.
\end{cases}
\]
In this case, we have the same desired characterization.
\end{itemize}
This completes the proof.\qed

\end{document}